\documentclass{amsart}
\usepackage{graphics}
\usepackage{amsfonts,amsmath}
\begin{document}

 \newtheorem{thm}{Theorem}[section]
 \newtheorem{cor}[thm]{Corollary}
 \newtheorem{lem}[thm]{Lemma}{\rm}
 \newtheorem{prop}[thm]{Proposition}

 \newtheorem{defn}[thm]{Definition}{\rm}
 \newtheorem{assumption}[thm]{Assumption}
 \newtheorem{rem}[thm]{Remark}
 \newtheorem{ex}{Example}\numberwithin{equation}{section}

\def\la{\langle}
\def\ra{\rangle}
\def\e{{\rm e}}
\def\x{\mathbf{x}}
\def\by{\mathbf{y}}
\def\bz{\mathbf{z}}
\def\F{\mathcal{F}}
\def\R{\mathbb{R}}
\def\T{\mathbf{T}}
\def\N{\mathbb{N}}
\def\K{\mathbf{K}}
\def\bK{\overline{\mathbf{K}}}
\def\Q{\mathbf{Q}}
\def\M{\mathbf{M}}
\def\O{\mathbf{O}}
\def\C{\mathbf{C}}
\def\P{\mathbf{P}}
\def\Z{\mathbf{Z}}
\def\H{\mathcal{H}}
\def\A{\mathbf{A}}
\def\V{\mathbf{V}}
\def\AA{\overline{\mathbf{A}}}
\def\B{\mathbf{B}}
\def\c{\mathbf{c}}
\def\L{\mathbf{L}}
\def\bS{\mathbf{S}}
\def\H{\mathcal{H}}
\def\I{\mathbf{I}}
\def\Y{\mathbf{Y}}
\def\X{\mathbf{X}}
\def\G{\mathbf{G}}
\def\f{\mathbf{f}}
\def\z{\mathbf{z}}
\def\y{\mathbf{y}}
\def\d{\hat{d}}
\def\bx{\mathbf{x}}
\def\y{\mathbf{y}}
\def\v{\mathbf{v}}
\def\g{\mathbf{g}}
\def\w{\mathbf{w}}
\def\b{\mathcal{B}}
\def\a{\mathbf{a}}
\def\b{\mathbf{b}}
\def\q{\mathbf{q}}
\def\u{\mathbf{u}}
\def\s{\mathcal{S}}
\def\cc{\mathcal{C}}
\def\co{{\rm co}\,}
\def\cop{{\rm COP}\,}
\def\tg{\tilde{f}}
\def\tx{\tilde{\x}}
\def\supmu{{\rm supp}\,\mu}
\def\supnu{{\rm supp}\,\nu}
\def\m{\mathcal{M}}
\def\s{\mathcal{S}}
\def\k{\mathcal{K}}
\def\la{\langle}
\def\ra{\rangle}

\title[nonnegativity]{Best $\ell_1$-approximation of
nonnegative polynomials by sums of squares}
\author{Jean B. Lasserre}
\address{LAAS-CNRS and Institute of Mathematics\\
University of Toulouse\\
LAAS, 7 avenue du Colonel Roche\\
31077 Toulouse C\'edex 4,France}
\email{lasserre@laas.fr}
\date{}

\begin{abstract}
Given a nonnegative polynomial $f$, we provide an explicit expression for its best $\ell_1$-norm approximation
by a sum of squares of given degree.
\end{abstract}

\keywords{Polynomials; sums of squares; semidefinite programming}

\subjclass{11E25 12D15 41A50 90C22}

\maketitle

\section{Introduction}

This note is concerned with the cone of nonnegative polynomials
and its subcone of polynomials that are sums of squares (s.o.s.). Understanding
the difference between these two cones is of practical importance
because if on the one hand nonnegative polynomials are ubiquitous, on the other hand
s.o.s. polynomials are much easier to handle.
For instance, and in contrast with nonnegative polynomials, checking whether a given polynomial is s.o.s.
can be done efficiently by solving a semidefinite program, a powerful technique of convex optimization.

A negative  result by Blekherman \cite{gary} states that when the degree is fixed,
there are much more nonnegative polynomials than sums of squares and the gap
between the two corresponding cones increases with the number of variables. On the other hand,
if the degree is allowed to vary, it has been known for some time that the cone of s.o.s. polynomials is dense (for
the $\ell_1$-norm of coefficients) in the cone of polynomials nonnegative on 
the box $[-1, 1]^n$. See e.g. Berg, Christensen and Ressel \cite{berg1} and Berg \cite{berg2}. However,
\cite{berg1} was essentially an existence result and subsequently, 
Lasserre and Netzer \cite{lassnet} have 
provided a very simple and explicit sequence of s.o.s. polynomials converging for the $\ell_1$-norm
to a given nonnegative polynomial $f$.

In this note we provide an explicit expression for the {\it best} $\ell_1$-norm approximation of a given nonnegative polynomial
$f\in\R[\x]$ by a s.o.s. polynomial $g$ of given degree $2d$ ($\geq{\rm deg}\,f$). It turns out that 
\[g\,=\,f+\lambda_0^*+\sum_{i=1}^n\lambda_i^* x_i^{2d},\]
for some nonnegative vector $\lambda^*\in\R^{n+1}$,
very much like the approximation already provided in \cite{lassnet} (where the $\lambda_i^*$'s are equal). In addition,
the vector $\lambda^*$ is an optimal solution of an explicit semidefinite program, and so can be computed efficiently.
 \section{Main result}
\subsection{Notation and definitions}
Let $\R[\x]$ (resp. $\R[\x]_d$) denote the ring of real polynomials in the variables
$\x=(x_1,\ldots,x_n)$ (resp. polynomials of degree at most $d$), whereas $\Sigma[\x]$ (resp. $\Sigma[\x]_d$) denotes 
its subset of sums of squares (s.o.s.) polynomials (resp. of s.o.s. of degree at most $2d$).
For every
$\alpha\in\N^n$ the notation $\x^\alpha$ stands for the monomial $x_1^{\alpha_1}\cdots x_n^{\alpha_n}$ and for every $i\in\N$, let $\N^{p}_d:=\{\beta\in\N^n:\sum_j\beta_j\leq d\}$ whose cardinal is $s(d)={n+d\choose n}$.
A polynomial $f\in\R[\x]$ is written 
\[\x\mapsto f(\x)\,=\,\sum_{\alpha\in\N^n}\,f_\alpha\,\x^\alpha,\]
and $f$ can be identified with its vector of coefficients $\f=(f_\alpha)$ in the canonical basis 
$(\x^\alpha)$, $\alpha\in\N^n$. 
Hence, denote by $\Vert f\Vert_1$ the $\ell_1$-norm $\sum_{\alpha}\vert f_\alpha\vert$
of the coefficient vector $\f$ which also defines a norm on $\R[\x]_d$.

Let $\s^p\subset\R^{p\times p}$ denote the space of real $p\times p$ symmetric matrices.
 For any two matrices $\A,\B\in\s^p$, the notation $\A\succeq0$ (resp. $\succ0$) stands for $\A$ is positive semidefinite
 (resp. positive definite), and the notation
$\la \A,\B\ra$ stands for ${\rm trace}\,\A\B$.

Let $\v_d(\x)=(\x^\alpha)$, $\alpha\in\N^n_d$, and let
$\B_\alpha\in\R^{s(d)\times s(d)}$ be real symmetric matrices such that
\begin{equation}
\label{balpha}
\v_d(\x)\,\v_d(\x)^T\,=\,\sum_{\alpha\in\N^n_{2d}}\x^\alpha\,\B_\alpha.\end{equation}
Recall that a polynomial $g\in\R[\x]_{2d}$ is a s.o.s. if and only if there exists a real positive semidefinite matrix
$\X\in\R^{s(d)\times s(d)}$ such that
\[g_\alpha\,=\,\langle \X,\B_\alpha\ra,\qquad \forall \alpha\in\N^n_{2d}.\]

\subsection*{$d$-moment matrix} With a sequence $\y=(y_\alpha)$, $\alpha\in\N^n$,
let $L_\y:\R[\x]\to\R$ be the linear functional
\[f\quad (=\sum_{\alpha}f_{\alpha}\,\x^\alpha)\quad\mapsto\quad
L_\y(f)\,=\,\sum_{\alpha}f_{\alpha}\,y_{\alpha},\quad f\in\R[\x].\]
With $d\in\N$, the $d$-moment matrix associated with $\y$
is the real symmetric matrix $\M_d(\y)$ with rows and columns indexed 
in $\N^n_d$, and defined by:
\begin{equation}
\label{moment}
\M_d(\y)(\alpha,\beta)\,:=\,L_\y(\x^{\alpha+\beta})\,=\,y_{\alpha+\beta},\qquad\forall\alpha,\beta\in\N^n_d.\end{equation}
It is straightforward to check that
\[\left\{\,L_\y(g^2)\geq0\quad\forall g\in\R[\x]_d\,\right\}\quad\Leftrightarrow\quad\M_d(\y)\,\succeq\,0,\quad d=0,1,\ldots.\]

\subsection*{Semidefinite programming} A semidefinite program is a convex (more precisely convex conic) optimization problem 
of the form $\min_{\X}\: \{\la \C,\X\ra\::\:\mathcal{A}\,\X\,=\b;\:\X\succeq0\}$,
for some real symmetric matrices $\C,\X\in\s^p$, vector $\b\in\R^m$,
and some linear mapping $\mathcal{A}:\s^p\to\R^m$. Semidefinite programming is a powerful technique 
of convex optimization, ubiquitous in many areas. A semidefinite program can be solved
efficiently and even in time polynomial in the input size of the problem, for fixed arbitrary precision.
For more details the interested reader is referred to e.g.
\cite{boyd}.

\subsection{The result}

Consider the following optimization problem:
\begin{equation}
\label{inv-1}
\rho_d:=\displaystyle\min_g\,\{\:\Vert f-g\Vert_1\::\: g\in\Sigma[\x]_d\:\},
\end{equation}
that is, one searches for the best $\ell_1$-approximation of $f$ by a s.o.s. polynomial of degree at most $2d\,(\geq{\rm deg}\,f)$.
Of course, and even though (\ref{inv-1}) is well defined for an arbitrary $f\in\R[\x]$,
such a problem is of particular interest when $f$ is nonnegative but not a s.o.s.
\begin{thm}
\label{thmain}
Let $f\in\R[\x]$ and let $2d\geq {\rm deg}\,f$.
The best $\ell_1$-norm approximation of $f$ by a s.o.s. polynomial of degree
at most $2d$ is given by
\begin{equation}
\label{best}
\x\,\mapsto\: g(\x)\,=\,f(\x)+\,(\lambda^*_0+\sum_{i=1}^n \lambda^*_i\,x_i^{2d}),
\end{equation}
for some nonnegative vector $\lambda^*\in\R^{n+1}$. Hence $\rho_d=\displaystyle\sum_{i=0}^n\lambda^*_i$, and 
in addition, $\lambda^*$ is an optimal solution of the semidefinite program:
\begin{equation}
\label{dual3}
\displaystyle\min_{\lambda\geq0}\:\left\{\displaystyle\sum_{i=0}^n\lambda_i\::\:
f+\lambda_0+\displaystyle\sum_{i=1}^n\lambda_i\,x_i^{2d}\in\Sigma[\x]_d\:\right\}.
\end{equation}
\end{thm}
\begin{proof}
Consider $f$ as an element of $\R[\x]_{2d}$ by setting $f_\alpha=0$ whenever
$\vert\alpha\vert >{\rm deg}\,f$ (where $\vert\alpha\vert=\sum\alpha_i$),
and rewrite (\ref{inv-1}) as the semidefinite program:

 \begin{equation}
\label{inv-3}
\begin{array}{rlll}
\rho_d:=\displaystyle\min_{\lambda\geq0,\X\succeq0,g}&\displaystyle\sum_{\alpha\in\N^n_{2d}}\lambda_\alpha&&\\
\mbox{s.t.}& \lambda_\alpha +g_\alpha&\geq f_\alpha,&\forall \alpha\in\N^n_{2d}\\
& \lambda_\alpha -g_\alpha&\geq -f_\alpha,&\forall \alpha\in\N^n_{2d}\\
&g_\alpha-\la\X,\B_\alpha\ra &=0,&\forall \alpha\in\N^n_{2d}.
\end{array}
\end{equation}
The dual semidefinite program of (\ref{inv-3}) reads:

\begin{equation}
\label{dual}
\left\{\begin{array}{rrll}
\displaystyle\max_{u_\alpha,v_\alpha\geq0,\y}&\displaystyle
\sum_{\alpha\in\N^n_{d}}f_\alpha(u_\alpha-v_\alpha)&&\\
\mbox{s.t.}&u_\alpha+v_\alpha&\leq \,1&\forall\alpha\in\N^n_{2d}\\
&u_\alpha-v_\alpha+y_\alpha &\,=\,0&\forall\alpha\in\N^n_{2d},\\
&\M_d(\y)&\succeq\,0,&
\end{array}\right.
\end{equation}
or, equivalently,
\begin{equation}
\label{dual1}
\left\{\begin{array}{ll}
\displaystyle\max_{\y}&-L_\y(f)\\
\mbox{s.t.}&\M_d(\y)\,\succeq\,0\\
&\vert y_\alpha\vert \,\leq\,1,\quad\forall\alpha\in\N^n_{2d}.
\end{array}\right.
\end{equation}
The semidefinite program (\ref{dual1}) has an optimal solution $\y^*$ because the feasible set is compact.
In addition,
let $\y=(y_\alpha)$ be the moment sequence of
the measure $d\mu=\e^{-\Vert\x\Vert^2}d\x$, scaled so that $\vert y_\alpha\vert <1$ for all $\alpha\in\N^n_{2d}$.
Then $(\y,\u,\v)$ with $\u=-\min[\y,0]$ and $\v=\max[\y,0]$, is strictly feasible in (\ref{dual}) 
because $\M_d(\y)\succ 0$, and so Slater's condition\footnote{Slater's condition holds the conic optimization problem $\P:\:\min_\x \{\c'\x\,:\, \A\x=\b;\:\x\in\K\}$, where $\K\subset\R^n$ is a convex cone, if
there exists a strictly feasible solution $\x_0\in{\rm int}\,\K$.
In this case, there is no duality gap
between $\P$ and its dual $\P^*:\:\max_\z\{\b'\z\,:\,\c-\A'\z\in\K^*\}$. 
In addition, if the optimal value is bounded then $\P^*$ has an optimal solution.}
holds for (\ref{dual}). 
Therefore, by a standard duality result in convex optimization, there is no duality gap
between (\ref{inv-3}) and (\ref{dual}) (or (\ref{dual1})), and (\ref{inv-3}) has an optimal solution $(\lambda^*,\X^*,g^*)$.
Hence $\rho_d=-L_{\y^*}(f)$ for any optimal solution $\y^*$ of (\ref{dual1}).

Now by \cite[Lemma 1]{arch}, $\M_d(\y)\succeq0$ implies that
$\vert y_\alpha\vert\leq\max[L_\y(1),\max_iL_\y(x_i^{2d})]$, for every $\alpha\in\N^n_{2d}$. Therefore,
(\ref{dual1}) has the equivalent formulation
\begin{equation}
\label{dual2}
\left\{\begin{array}{rrl}
\rho_d=-\displaystyle\min_{\y}&L_\y(f))&\\
\mbox{s.t.}&\M_d(\y)&\succeq\,0\\
&L_\y(1)&\leq\,1\\
&L_\y(x_i^{2d})&\leq\,1,\quad i=1,\ldots,n,
\end{array}\right.
\end{equation}
whose dual is eaxctly (\ref{dual3}).
Again Slater's condition holds for (\ref{dual2}) and it has an optimal solution $\y^*$.
Therefore (\ref{dual3}) also has an optimal solution $\lambda^*\in\R^{n+1}_+$ with
$\rho_d=\sum_i\lambda_i^*$, the desired result.
\end{proof}
So the best $\ell_1$-norm s.o.s. approximation $g$ in Theorem (\ref{thmain}) is very much the same as the 
$\ell_1$-approximation provided in Lasserre and Netzer \cite{lassnet} where
all coefficients $\lambda^*_j$ were identical.
\begin{ex}
{\rm Consider the Motzkin-like polynomial\footnote{Computation was made by running the GloptiPoly software \cite{didier}
dedicated to solving the generalized problem of moments.}  $\x\mapsto f(\x)=x_1^2x_2^2(x_1^2+x_2^2-1)+1/27$ of degree $6$, which is nonnegative but not  
a s.o.s., and with a global minimum $f^*=0$ attained at  4 global minimizers $\x^*=(\pm (1/3)^{1/2}, \pm (1/3)^{1/2})$.
The results are displayed in Table \ref{tab1} for $d=3,4,5$.

\begin{table}[ht]
\label{tab1}
\begin{center}
\begin{tabular}{|| l | l | l ||}
\hline
$d$ & $\lambda^*$  & $\rho_d$ \\
\hline
\hline
$3$ &  $\approx 10^{-3}\,(5.445,  5.367 , 5.367)$ & $\approx 1.6\, 10^{-2}$\\
$4$& $\approx 10^{-4}\,(2.4 ,  9.36 , 9.36)$ &$\approx 2.\,10^{-3}$\\
$5$& $\approx 10^{-5}\,(0.04 ,  4.34, 4.34)$ &$\approx 8.\,10^{-5}$\\
\hline 
\end{tabular}
\end{center}
\caption{Best $\ell_1$-approximation for the Motzkin polynomial.}
\end{table}
}

\end{ex}

\end{document}